\documentclass{amsart}

\usepackage{subcaption}
\usepackage{amsmath}
\usepackage{amssymb}
\usepackage{graphicx}
\usepackage{enumitem}
\usepackage{comment}
\usepackage{booktabs}
\usepackage{subcaption}
\usepackage{url}
\usepackage{multicol}
\usepackage{epstopdf}

\newtheorem{thm}{Theorem}
\newtheorem{lem}{Lemma}

\newtheorem{prop}{Proposition}

\newtheorem{definition}[equation]{Defintion}

\theoremstyle{definition}
\newtheorem{ex}{Example}

\newtheorem*{prob*}{Problem}

\DeclareMathOperator{\End}{End}

\DeclareMathOperator{\Adj}{Adj}
\DeclareMathOperator{\Der}{Der}
\DeclareMathOperator{\Cent}{Cent}
\DeclareMathOperator{\GL}{{\rm GL}}

\DeclareMathOperator{\im}{im}
\DeclareMathOperator{\sign}{sign}

\setlist[description]{leftmargin=\parindent,labelindent=\parindent}

\title{Efficient characteristic refinements for finite groups}
\author{Joshua Maglione}
\email{maglione@math.colostate.edu}
\address{
	Department of Mathematics,
	Colorado State University,
	Fort Collins, CO 80523,
	USA
}

\begin{document}

\begin{abstract}
Filters were introduced by J.B. Wilson in 2013 to generalize work of Lazard with associated graded Lie rings.
It holds promise in improving isomorphism tests, but the formulas introduced then were impractical for computation.
Here, we provide an efficient algorithm for these formulas, and we demonstrate their usefulness on several examples of $p$-groups.
\end{abstract}

\maketitle

\section{Introduction}

Isomorphism between two finite groups becomes easier when we use isomorphism invariant subgroups (i.e. characteristic subgroups) to constrain the number of possibilities. 
With this in mind, Fitting uncovered several characteristic subgroups to later be used to determine isomorphism between groups \cite{Fitting}, see the accompanying bibliography in \cite{CH:isomorphism}.
However, in the case of $p$-groups, these characteristic subgroups are usually the whole group or the trivial group.
As seen in \cite{ELGO:Auts}, the inclusion of just one new characteristic subgroup can greatly improve performance.

New sources for computable characteristic subgroups of $p$-groups were uncovered in \cite{W:Char,W:filters}.
In addition, it was shown that the inclusion of new characteristic subgroups induced more subgroups and gave formulas to automate this process of refining.
However, the formulas required an exponential amount of computation. 
In this paper, we prove that we can do this in polynomial time and we provide an implementation for {\sc Magma}. 
Indeed, even for groups of order $3^{100}$, we are able to refine a typical characteristic series by about ten-fold in just a few minutes; see Figure \ref{fig:RandomGroups} on page \pageref{fig:RandomGroups}.

A \emph{filter} for a group $G$ is a function $\phi : M\rightarrow 2^G$ from a commutative monoid $M=\langle M,+,0,\preceq\rangle$ into the normal subgroups of $G$ satisfying the following: for all $s,t\in M$
\[ [\phi_s,\phi_t]\leq \phi_{s+t} \qquad\&\qquad s\preceq t \Longrightarrow \phi_t\leq \phi_s.\]
Wilson proves \cite[Theorem 3.1]{W:Char} that each filter has an associated Lie ring:
\begin{equation}\label{eqn:Lie-ring} 
L(\phi) = \bigoplus_{s\in M-\{0\}} \phi_s/ \left\langle \phi_{s+t} \mid t\in M-\{0\}\right\rangle.
\end{equation}
The use of monoids $M$ is essential as it allows for somewhat arbitrary refinements some of which are discussed in Section \ref{refinements}.
We prove the following theorem.

\begin{thm}\label{thm:main}
Suppose $\phi:\mathbb{N}^d\rightarrow 2^G$ is a filter where $\preceq$ is the lexicographical order. 
If $H\triangleleft G$ and there exists $s\in \mathbb{N}^d$ such that 
\[ \left\langle \phi_{s+t} \mid t\in \mathbb{N}^d-\{0\}\right\rangle < H < \phi_s,\] 
then there exists a polynomial-time algorithm that refines $\phi$ to contain $H$ in its image.
\end{thm}

The result is smaller homogeneous components, faster automorphism computations, and an easier explanation of structure. 
Indeed, in \cite{M:classical}, it was shown that even well-studied unipotent classical groups admit surprises such as characteristic filters whose factors are at most of order $p^2$.
Together with \cite{ELGO:Auts}, this then reduces automorphism questions to $\GL(2,p)$ instead of $\GL(d,p)$.
This and further uses in \cite{M:SmallAuts,W:filters} make it desirable to compute with filters efficiently.

In addition to providing a computational framework for filters in Section \ref{algorithms}, we refine several filters for common examples of $p$-groups in Section \ref{examples}.
We look to large examples in the literature and we also consider a sample of 2,000 sections (i.e. quotients of subgroups) of the Sylow 3-subgroups of classical groups of Lie type.
We find that the larger the group, the more new structure we find, and because of the repetitive nature, often one discovery leads to more discoveries.
All of our computations were run in {\sc Magma} V.21-5 \cite{Magma} on a computer with Intel Xeon W3565 microprocessors at 3.20 GHz.

\section{Preliminaries}

We denote the set of nonnegative integers by $\mathbb{N}$, and the set of all subsets of a set $G$ by $2^G$. For groups and rings, we follow notation found in \cite{G:GT}. For $g,h\in G$, we set 
\[ [g,h]=g^{-1}g^{h}=g^{-1}h^{-1}gh;\] 
for $X,Y\subseteq G$, we set 
\[ [X,Y]=\langle [x,y] : x\in X, y\in Y\rangle.\]
We let $\mathbb{Z}_p$ denote the group $\mathbb{Z}/p\mathbb{Z}$.

For a $p$-group $G$, we consider two recursively defined series: the lower central series and the exponent-$p$ central series. 
The lower central series starts with $\gamma_1(G)=G$ and $\gamma_{i+1}(G)=[\gamma_i(G),G]$, and the exponent-$p$ central series
begins with $\eta_1(G)=G$ and $\eta_{i+1}(G) = [\eta_i(G),G]\eta_i(G)^p$. The class ($p$-class) of $G$ is the number of nontrivial terms in the lower central series (exponent-$p$ central series).

\subsection{Complexity}

An algorithm runs in \emph{polynomial time} if the number of operations it uses is bounded by a polynomial of the input length.
At least one  mark of efficiency is polynomial time, but we include run times from experiments as well.

We assume the standard models of computation for groups: permutation and matrix groups and groups given by finite presentations, see \cite{HEO:Handbook}.
All of our methods use at most a polynomial number of operations in the input size, which can be as small as $\log|G|$.
They further depend on efficient methods to compute the order of subgroups and normal closures. 
This is in polynomial time for groups represented as permutations or matrices \cite{S:Perms,L:solvable-matrix}.
For groups given by power-conjugate presentations, there are highly practical algorithms for these tasks, although some problems are not known to be in polynomial time, see \cite{LGS:left,LGS:deep-thought}.
We remark that all of our tests used power-conjugate presentations.

\subsection{Filters and prefilters}

The formulas found in \cite{W:Char} are for commutative monoids, but for computational feasibility, we concentrate only on $\mathbb{N}^d$ with the lexicographical order.
We note that every filter $\phi:\mathbb{N}^d\rightarrow 2^G$ induces a \emph{boundary filter} $\partial\phi:\mathbb{N}^d\rightarrow 2^G$, where $\partial\phi_s = \langle \phi_{s+t} \mid 0\prec t\rangle$.
These are the factor groups of the homogeneous components of the Lie ring in equation (\ref{eqn:Lie-ring}).

\begin{definition}
A function $\pi : X\rightarrow 2^G$ is a \emph{prefilter} if it satisfies the following conditions.
\begin{enumerate}
\item There exists $d$ such that $0\in X\subseteq \mathbb{N}^d$ and $\langle X \rangle = \mathbb{N}^d$;
\item if $x\in X$ and $y\in \mathbb{N}^d$ with $y\preceq x$, then $y\in X$;
\item for all $x\in X$, $\pi_x\trianglelefteq G$;
\item if $x,y\in X$ with $x\preceq y$ then $\pi_x\geq \pi_y$.
\end{enumerate}
\end{definition}

For $s\in \langle X\rangle$, a \emph{partition} of $s$ with respect to $X$ is a sequence $(s_1,...,s_k)$ where each $s_i\in X$ and $s = \sum_{i=1}^ks_i$. 
Let $\mathcal{P}_X(s)$ denote the set of partitions of $s\in \langle X\rangle$ with respect to $X$, and if $P=(s_1,...,s_k)\in\mathcal{P}_X(s)$, then set
\[ [\pi_P ] = [ \pi_{s_1}, ... ,\pi_{s_k} ].\] 
For a function $\pi : X\rightarrow 2^G$, define a new function $\overline{\pi}:\langle X\rangle \rightarrow 2^G$ where 
\begin{equation}\label{eqn:closure} 
\overline{\pi}_s = \prod_{P\in\mathcal{P}_X(s)} [\pi_P].
\end{equation}
\begin{thm}[{\cite[Theorem 3.3]{W:Char}}]
If $\pi$ is a prefilter, then $\overline{\pi}$ is a filter.
\end{thm}

Observe that $|\mathcal{P}_X(s)|$ is exponential in $\log|G|$, so the equation for $\overline{\pi}$ in (\ref{eqn:closure}) is not a practical nor a polynomial-time formula.

\section{Algorithms}\label{algorithms}

We have two basic problems for filters.
We remind the reader that all of our filters are totally ordered, and we assume that $\mathbb{N}^d$ is ordered by the lexicographical ordering.

\begin{prob*}{\sc Evaluate}
\begin{description}
\item[Input:] A filter $\phi : \mathbb{N}^d\rightarrow 2^G$ and $s\in\mathbb{N}^d$;
\item[Return:] Generators for the subgroup $\phi_s$.
\end{description}
\end{prob*}

\begin{prob*}{\sc Boundary}
\begin{description}
\item[Input:] A filter $\phi : \mathbb{N}^d\rightarrow 2^G$;
\item[Return:] A filter $\partial\phi : \mathbb{N}^d\rightarrow 2^G$.
\end{description}
\end{prob*}

After we determine polynomial-time algorithms for these problems, we discuss an algorithm for the following problem and prove it can be solved in polynomial time.

\begin{prob*}{\sc Generate}
\begin{description}
\item[Input:] A prefilter $\pi : \mathbb{N}^d\rightarrow 2^G$;
\item[Return:] Its closure $\overline{\pi} : \mathbb{N}^d\rightarrow 2^G$.
\end{description}
\end{prob*}

\subsection{A data structure of filters}

We introduce a data structure to compute with filters which admits polynomial-time algorithms for the three problems listed above.
To do this, we assume the following properties of our filters.
\begin{definition}
A filter $\phi:\mathbb{N}^d\rightarrow 2^G$ is \emph{full} if for all $1\ne H\in\im(\phi)$, there exists $s\in \mathbb{N}^d$ such that $\phi_s=H$ and $\phi_s\ne\partial\phi_s$.
\end{definition} 
Said another way, we assume that $I_H=\{ s\in \mathbb{N}^d : \phi_s=H\}$ has a maximal element.
Since we assume a total ordering, $I_H$ has a unique maximal element, and we denote it by $m_H$.
Observe that a filter $\phi:\mathbb{N}^d\rightarrow 2^G$ for a finite $p$-group $G$ is full if, and only if, $|\partial\phi_0|=|L(\phi)|$.
This might seem like a restrictive assumption, but it is relatively harmless.
We will prove that for every filter, we can construct a full filter without changing the image in Section \ref{full-filters}.

The data structure of a filter $\phi:\mathbb{N}^d\rightarrow 2^G$ is a signed set that stores the pairs $(m,H)\in \mathbb{N}^d\times 2^G$ for each $1\ne H\in\im(\phi)$.
If for every $1\ne H\in\im(\phi)$, the set $I_H$ has a maximum, then sign$(\phi)=1$ and the stored index for $H$ will be its maximum index.
Otherwise, sign($\phi)=-1$, and the stored indices are minimal indices.
Therefore, for computation, a filter $\phi:\mathbb{N}^d\rightarrow 2^G$ is regarded as
\[ \phi = (\text{sign}(\phi), \{ (m,H) : 1\ne H\in \im(\phi) \} ).\]

Since all of our filters are series, if $1\ne H\in\im(\phi)$, define $H^+$ to be the next term in the descending series.
Therefore, $H>H^+$.
We first prove a lemma that characterizes when $I_H$ has a maximal element.

\begin{lem}\label{lem:max}
Suppose $\phi : \mathbb{N}^d\rightarrow 2^G$ is a filter and $1\ne H\in\im(\phi)$. Let $m=\min(I_{H^+})$. The set $I_H$ has a maximum element if, and only if, the $d$th entry of $m$ is nonzero.
\end{lem}

\begin{proof}
Since $\phi$ is a filter it follows that for all $s\in I_H$, $s\prec m$.
Suppose the $d$th entry of $m$ is 0. Choose $s\in I_H$; then each term in the increasing sequence 
\[ s\prec s + e_d \prec s+2e_d \prec \cdots \]
is strictly less than $m$. Therefore, $I_H$ cannot have a maximum. On the other hand, suppose the $d$th entry is nonzero. Then $m-e_d\in\mathbb{N}^d$. Since $m-e_d\prec m=\min(I_{H^+})$, it follows that $m-e_d\in I_H$, and thus, $\max(I_H)=m-e_d$. 
\end{proof}

\begin{prop}
{\sc Evaluate} and {\sc Boundary} are in polynomial time.
\end{prop}

\begin{proof}
We first solve {\sc Evaluate}. 
Because $\preceq$ is a total order, $\mathbb{N}^d$ is partitioned into intervals $I_H$ and $\phi$ stores $O(\log|G|)$ pairs $(m,H)$.
Thus, given $s\in \mathbb{N}^d$, find the interval $I_H$ such that $s\in I_H$. 
That is, find consecutive terms in $\phi$, $(m,H)$ and $(n,K)$, such that $m \prec s \prec n$.
Therefore, {\sc Evalute} requires $O(\log\log|G|)$ applications of $\preceq$ in $\mathbb{N}^d$.

Because of Lemma \ref{lem:max}, {\sc Boundary} is solved in constant time by the following constructions.
If $\text{sign}(\phi)=1$, then $\partial\phi = (-1, \{ (0,G)\} \cup \{ (m_H,H^+) : 1\ne H\in\im(\phi)\} )$.
If $\text{sign}(\phi)=-1$, then define $n_G=0$ and, for $1\ne H\in\im(\phi)$,
\[ n_{H^+} = \left\{ \begin{array}{ll} m_H & \text{if } m_H \text{ exists}, \\ \min(I_{H^+}) & \text{otherwise} .\end{array}\right.\]
Therefore, $\partial\phi = (-1,\{ (n_H,H) : H \in \im(\phi) \} )$.
\end{proof}

\subsection{Full filters}\label{full-filters}

We assume $\phi:\mathbb{N}^d\rightarrow 2^G$ is a filter that is not full.
We show that we can always fill our filters to make them full.

\begin{lem}\label{lem:fill}
Suppose $\phi:\mathbb{N}^d\rightarrow 2^G$ is a filter that is not full. 
If $H\ne 1$ is the largest subgroup in $\im(\phi)$ such that $I_H$ does not have a maximal element, then there exists a filter $\gamma:\mathbb{N}^d\rightarrow 2^G$ with the following properties: 
\begin{enumerate}
\item $\im(\gamma)=\im(\phi)$,
\item if $K\in\im(\phi)$ such that $I_K$ has a maximum, then $\{ s\in\mathbb{N}^d\mid \gamma_s=K\}$ has a maximum, and 
\item $\{ s\in\mathbb{N}^d \mid \gamma_s=H\}$ has a maximum.
\end{enumerate}
\end{lem}

\begin{proof}
Define 
\begin{equation}\label{eqn:effective-max} 
e = \max(\{ m_X + m_Y \mid X,Y\in\im(\phi), X>H, Y>H \}\cap I_H).
\end{equation}
In case the intersection in (\ref{eqn:effective-max}) is trivial, set $e=\min(I_H)$. 
Otherwise it is finite and must contain a maximum.
Suppose $s,t\in \mathbb{N}^d$ such that $s+t=e$. 
Suppose $s\in I_X$, $t\in I_Y$ and if $s,t$ are not maximal, then this would contradict (\ref{eqn:effective-max}).
Since both $s$ and $t$ are maximal, it follows that 
\[ [\partial\phi_s,\phi_t] < [\phi_s,\phi_t] \leq \phi_e = H.\]
Hence, for all $s,t\in \mathbb{N}^d$, where $s+t=e$, it follows that $[\partial\phi_s,\phi_t]\leq H^+$.

For each $X\in\im(\phi)-\{H,H^+\}$ define $J_X=I_X$. Additionally, set
\[ J_H =\{ s\in I_H \mid s\preceq e\} \quad\text{and}\quad J_{H^+} =\{ s \in I_H \mid e \prec s \} \cup I_{H^+}.\]
For each $s\in \mathbb{N}^d$, there exists a unique $X\in\im(\phi)$ such that $s\in J_X$, so define $\gamma_s=X$.
Since $\phi$ is a filter, it follows that $\gamma_s\trianglelefteq G$ for all $s\in\mathbb{N}^d$. 
Moreover, $\gamma$ is order reversing.

Let $s,t\in\mathbb{N}^d$. 
There are a few cases depending on whether $s$, $t$, or $s+t$ are contained in $\{ u\in I_H \mid e\prec u\}$.
If only $s+t\in \{ u\in I_H \mid e\prec u\}$, then, using (\ref{eqn:effective-max}), $[\gamma_s,\gamma_t] = [\phi_s,\phi_t]<H$. 
Therefore, $[\gamma_s,\gamma_t]\leq H^+=\gamma_{s+t}$.
All the other cases use similar arguments. Thus, $\gamma$ is a filter, and the lemma follows.
\end{proof}

\begin{thm}\label{thm:full}
Suppose $\phi:\mathbb{N}^d\rightarrow 2^G$ is a filter. There exists polynomial-time algorithms that
\begin{enumerate}
\item decide if $\phi$ is full, and 
\item if $\phi$ is not full, construct a full filter $\gamma:\mathbb{N}^d\rightarrow 2^G$ such that $\im(\gamma)=\im(\phi)$.
\end{enumerate}
\end{thm}

\begin{proof}
For (1), this follows from Lemma \ref{lem:max}. 
For (2), iterate Lemma \ref{lem:fill} until every $I_H$ has a maximal element. 
This is done with $O(\log^3|G|)$ operations in $\mathbb{N}^d$ and applications of $\preceq$.
\end{proof}

\subsection{Generating filters from prefilters}

Next, we generate filters from prefilters $\pi:X\rightarrow 2^G$. 
We store prefilters in the same way we store filters, so in terms of data structures, filters and prefilters are indistinguishable. 
We note that we only consider the case where $\sign(\pi)=1$.
Because prefilters are not required to satisfy $[\pi_s,\pi_t]\leq \pi_{s+t}$ one can simply change the function of $\pi$ to allow for maximal indices. 
However, prefilters often come as refinements of filters, so to keep the filter structure intact, one employs Theorem \ref{thm:full}.

We prove a useful lemma first.

\begin{lem}\label{lem:max-index}
Let $\pi:X\rightarrow 2^G$ be a prefilter with $X=\mathbb{N}^d$. If, for every $1\ne H\in\im(\pi)$, the set $I_H$ has a maximum, then the same holds for all $1\ne K\in\im(\overline{\pi})$.
\end{lem}

\begin{proof}
Let $H\in\im(\overline{\pi})$ and $s\in\mathbb{N}^d$ such that $\overline{\pi}_{s}=H$. 
Since $G$ is finite, there exists a finite and minimal $\mathcal{X}\subset \mathcal{P}_X(s)$ such that 
\[ H=\prod_{P\in\mathcal{X}}[\pi_P].\]
Let $\mathcal{X}'$ be the set of all partitions $P\in\mathcal{X}$, where each $s_i$ is replaced by its corresponding maximal index. 
That is, if $P=(s_1,...,s_k)\in\mathcal{X}$ and $H_i=\pi_{s_i}$, then define $P'=(\max(I_{H_i}))_{i=1}^k$, so $\mathcal{X}'=\{ P' : P \in\mathcal{X}\}$.
Observe that $|\mathcal{X}|=|\mathcal{X}'|$ and 
\[ H=\prod_{P\in\mathcal{X}}[\pi_P]=\prod_{P'\in\mathcal{X}'}[\pi_{P'}];\]
however, $\mathcal{X}'$ need not be contained in $\mathcal{P}_X(s)$.

Define 
\[ t = \min_{(t_1,...,t_k)\in\mathcal{X}'}\left(\sum_{i=1}^kt_i\right).\]
Such a minimum exists since $\mathbb{N}^d$ is totally ordered.
For $P=(t_1,...,t_k)\in\mathcal{X}'$, let $u=\sum_it_i$.
Because $t\preceq u$, it follows that $\overline{\pi}_t\geq \overline{\pi}_u$.
Therefore, $\overline{\pi}_{t}=H$. 

Let $u\in\mathbb{N}^d$ such that $t\prec u$. 
Define 
\[ \mathcal{Y}= \left\{ P\in\mathcal{X}' : u\preceq \sum_it_i\right\}\quad \text{and} \quad \mathcal{Z}= \left\{ P\in\mathcal{X}' : u\succ \sum_it_i \right\};\] 
note that $\mathcal{Z}$ is nonempty. If 
\[ H = \prod_{P\in\mathcal{Y}} [\pi_P] ,\]
then the partitions in $\mathcal{Z}$ are superfluous. 
Thus, there exists a smaller set of partitions $\mathcal{Y}$, which contradicts the minimality of $\mathcal{X}$. 
Hence, $t$ is the maximal index for $H$.
\end{proof}

Now we provide a polynomial-time algorithm for computing filters from prefilters, and hence prove Theorem \ref{thm:main}.

\begin{thm}\label{thm:poly}
There exists a polynomial-time algorithm for {\sc Generate}.
\end{thm}

\emph{Algorithm}. 
We perform a transitive closure.
Start with $\mathcal{S}=\{ (m_H, H ) : 1\ne H \in\im(\pi)\}$. 
We loop through all possible pairs $(x,H),(y,K)\in\mathcal{S}$ such that $[H,K]$ has not previously been computed, selecting pairs where $x+y$ is minimized.
Suppose $(x,H)$ and $(y,K)$ are such pairs.
Let $s= x + y$, and let $(t,L)\in\mathcal{S}$ such that $t$ is the smallest index with $s\preceq t$.
Now we update $\mathcal{S}$; set
\[ \mathcal{S} = \left\{ \begin{array}{ll} (\mathcal{S}-\{ (s,L)\} )\cup \{ (s,[H,K]L)\} & \text{if } s=t,\\
\mathcal{S} \cup \{ (s,[H,K]L)\} & \text{if } s\ne t.\end{array}\right.\]
Furthermore, for every $(u,X) \in\mathcal{S}$ where $u\prec s$, set
\[ \mathcal{S} = (\mathcal{S}-\{ (u,X) \}) \cup \{ (u,X[H,K]L) \}.\]

At this stage in the loop, it is possible to have duplicate groups in $\mathcal{S}$. 
We merge them, keeping the largest index, using the order of the subgroup to determine if the groups are equal. 

Now we are back to searching for pairs $(x,H),(y,K)\in\mathcal{S}$ such that $[H,K]$ has not been computed previously.
If there is another pair, then we remain in the loop.
Otherwise, we are done, and we return the filter $(1,\mathcal{S})$.\qed

\emph{Correctness}. 
Let $\phi=$ {\sc Generate}$(\pi)$. Let $s\in\mathbb{N}^d$ and set $H=\overline{\pi}_s$.
There exists a minimal (and finite) $\mathcal{X}\subset \mathcal{P}_{\mathbb{N}^d}(s)$ such that
\[ H = \prod_{P\in \mathcal{X}}[\pi_P].\]
If $P=(s_1,...,s_k)\in\mathcal{X}$, then for $1<j<k$, we assume $[\pi_{s_1},...,\pi_{s_j}]\not\in\im(\pi)$. 
Otherwise, replace $P$ with $(s_1+\cdots +s_j,s_{j+1},...,s_k)$.

At the start of the algorithm, $X=[\pi_{s_1},\pi_{s_2}]$ has not been computed, so $(s_1+s_2,X)$ gets inserted into $\mathcal{S}$.
Therefore, for $j\geq 2$, 
\[ (s_{j+1},\pi_{s_{j+1}}),(s_1+\cdots +s_j,[\pi_{s_1},...,\pi_{s_j}])\in\mathcal{S},\] 
but $[[\pi_{s_1},...,\pi_{s_j}],\pi_{s_{j+1}}]$ has not been computed. 
Thus, insert 
\[ (s_1+\cdots +s_{j+1}, [\pi_{s_1},...,\pi_{s_{j+1}}])\] 
into $\mathcal{S}$.
Therefore, $[\pi_P]$ gets computed by the algorithm.
Hence, $\im(\phi)=\im(\overline{\pi})$.

Fix $H\in\im(\phi)$. Let $s\in\mathbb{N}^d$ be maximal such that $\phi_s=H$. 
By Lemma \ref{lem:max-index}, there exists a maximal $t\in\mathbb{N}^d$ such that $\overline{\pi}_t=H$. 
By the definition of $\overline{\pi}$ from equation (\ref{eqn:closure}), it follows that $t\preceq s$.
However, for every $P\in\mathcal{P}_{\mathbb{N}^d}(t)$, the algorithm computes the group $[\pi_P]$, so $s\preceq t$.
Therefore, $\phi=\overline{\pi}$.\qed

\emph{Timing}. Since $\phi$ is totally ordered, $|\im(\phi)|\leq \log|G|$. 
Hence, we compute $O(\log^2|G|)$ commutator subgroups and the orders of $O(\log^2|G|)$ subgroups.\qed

\subsection{An example}

We demonstrate how the algorithm generates a filter from a prefilter. 
Suppose $G$ is the group of upper unitriangular $5\times 5$ matrices over the finite field $\mathbb{Z}_p$. 
Let $\gamma: \mathbb{N}\rightarrow 2^G$ be the filter obtained from the lower central series of $G$, where $\gamma_0=G$.
Note that $G$ is generated by $g_i = I_d + E_{i,i+1}$ for $1\leq i \leq 4$; 
$E_{ij}$ is the matrix with $1$ in the $(i,j)$ entry and $0$ elsewhere.

$G$ has a characteristic subgroup $H=\langle g_1, g_4, \gamma_2\rangle$, where $\gamma_1>H>\gamma_2$.
We construct a prefilter from $\gamma$ to include $H$. 
Define $\pi:\mathbb{N}^2\rightarrow 2^G$ where 
\begin{equation}\label{eqn:pi} 
\pi = \left(1, \left\{ \big( (1,0), G \big), \big( (1,1), H \big), \big( (2,0), \gamma_2 \big), \big( (3,0), \gamma_3 \big), \big( (4,0), \gamma_4 \big) \right\} \right).
\end{equation}

Now we want to construct $\phi=$ {\sc Generate}$(\pi)$.
We initialize $\mathcal{S}$ to be the set $\pi$ given in equation (\ref{eqn:pi}).
We run through all pairs in $\mathcal{S}$. 
The first pair, $\big((1,0),G\big)$ and $\big((1,0),G\big)$, provides no new information.
The next pair is $\big((1,0),G\big)$ and $\big((1,1),H\big)$.
Set $s=(2,1)$, so that $t=(3,0)$. 
Since $s\ne t$, we include the new subgroup $X=[G,H]\gamma_3$ in $\mathcal{S}$. 
Thus, 
\[ \mathcal{S} = \mathcal{S} \cup \left\{ \big( (2,1), X \big) \right\}.\]
We have no duplicate groups in $\mathcal{S}$ as $X\ne \gamma_i$ for $i\in\{2,3\}$. 
Therefore, we continue looping through pairs.

The next pair to consider is $\big((1,1),H\big)$ and $\big((1,1),H\big)$.
Therefore $s=(2,2)$ and so $t=(3,0)$. Since $s\ne t$, we include the subgroup $Y=[H,H]\gamma_3$ into $\mathcal{S}$, so
\[ \mathcal{S} = \mathcal{S} \cup \left\{ \big( (2,2), Y \big) \right\}.\]
Since $Y=\gamma_3$, we have duplicate groups in $\mathcal{S}$: $\big((2,2),\gamma_3\big)$ and $\big((3,0),\gamma_3\big)$.
Because $(2,2)\prec (3,0)$, we remove the entry with $(2,2)$ and only keep the entry with $(3,0)$.
Therefore, at this stage,
\begin{equation*}
\mathcal{S} = \left\{ \big( (1,0), G\big), \big( (1,1), H\big), \big( (2,0), \gamma_2\big), \big( (2,1), X\big), \big( (3,0), \gamma_3\big), \big( (4,0), \gamma_4\big) \right\}.
\end{equation*}

The next pair to consider is $\big((1,0),G\big)$ and $\big((2,0),\gamma_2\big)$ which, again, results in no new information.
Computing the commutator of $H$ with $\gamma_2$ yields $\big( (3,1), \gamma_3 \big)$.
This gets included in $\mathcal{S}$ because $(3,1)$ is not already included, but $\big( (3,0), \gamma_3 \big)$ and $\big( (3,1), \gamma_3 \big)$ are duplicate groups.
We remove $\big( (3,0), \gamma_3 \big)$ from $\mathcal{S}$.

The remaining computations provide no new subgroups, but continue to update the indices. The resulting set is 
\[ \mathcal{S} = \left\{ \big( (1,0), G\big), \big( (1,1), H\big), \big( (2,0), \gamma_2\big), \big( (2,1), X\big), \big( (3,1), \gamma_3\big), \big( (4,2), \gamma_4\big) \right\}. \]

\section{Refinements}\label{refinements}

In \cite{W:Char,W:filters}, sources of new subgroups were suggested. 
We use these in our testing, but we remark that our methods can use any source of refinements, and hence, any prefilter.

Suppose we start with the filter $\eta:\mathbb{N}\rightarrow 2^G$ given by the exponent $p$-central series of $G$. 
Then $L(\eta)$ has an associated $\mathbb{N}$-graded Lie algebra, which yields $\mathbb{Z}_p$-bilinear maps from the graded product (e.g.\!\! $[,] : L_s\times L_t \rightarrowtail L_{s+t}$). 
We turn to some associated algebras for these bilinear maps. 
Suppose $\circ : U \times V\rightarrowtail W$ is a biadditive map of abelian groups; define the adjoint, centroid, derivation, left scalar, and right scalar rings as
\begin{align*}
\begin{split}
\Adj( \circ ) &= \{ (f,g) \in \End(U)\times \End(V)^{\text{op}} : \forall u\in U, \forall v\in V, (uf) \circ v = u \circ (gv) \},\\
\Cent( \circ ) &= \{ (f,g,h) \in \End(U)\times \End(V)\times \End(W) : \forall u\in U, \forall v\in V, \forall w\in W, \\
&\qquad (uf) \circ v = u \circ (vg) = (u\circ v)h \}, \\
\Der( \circ ) &= \{ (f,g,h) \in \mathfrak{gl}(U)\times \mathfrak{gl}(V)\times \mathfrak{gl}(W) : \forall u\in U, \forall v\in V, \forall w\in W, \\
&\qquad (uf) \circ v + u \circ (vg) = (u\circ v)h \},\\
\mathcal{L}(\circ) &= \{ (f,g) \in \End(U)^{\text{op}}\times \End(W)^{\text{op}} : \forall u\in U, \forall w\in W, (fu) \circ v = g(u \circ v) \}, \text{ and }\\
\mathcal{R}(\circ) &= \{ (f,g) \in \End(V)\times \End(W) : \forall v\in V, \forall w\in W, u\circ (vf) = (u\circ v)g \}.
\end{split}
\end{align*}

It is in these nonassociative rings we begin to find more characteristic structure in $G$ \cite{W:Char, W:filters}. 
For example, the Jacobson radical acts on the homogeneous components and yields characteristic subgroups (for $\Der(\circ)$, this is done in the associative enveloping algebra).

\section{Examples}\label{examples}

\begin{ex}
We consider a $p$-group found in \cite[Section 12.1]{ELGO:Auts} and refine its lower central series by adding 5 subgroups. 
Define $G$ by the following power-commutator presentation where the missing commutator relations are assumed to be trivial
\begin{align*} 
G = \langle g_1,...,g_{13} &\mid [g_{10},g_6]=g_{11}, [g_{10},g_7] = g_{12},\\
&\;\;\; [g_2,g_1]=[g_4,g_3]=[g_6,g_5]=[g_8,g_7]=[g_{10},g_9]=g_{13}, \text{exponent }p\rangle.
\end{align*}

Let $L(\gamma)$ be the $\mathbb{N}$-graded Lie algebra associated to the lower central series of $G$. 
Let $\circ : L_1\times L_1\rightarrowtail L_2$ be the nontrivial graded product in $L(\gamma)$.
The adjoint algebra $A$ of $\circ$ is a 53 dimensional algebra with a nontrivial Jacobson radical $J$ of dimension 35.
Table \ref{tab:char} shows the number of new subgroups added, and Table \ref{tab:filter} shows the resulting filter.

\begin{table}[h]
\centering
\begin{tabular}{cccc}
$i$ & $\dim(J^i/J^{i+1})$ & $\dim(L_1J^i / L_1J^{i+1})$ & $[\phi_{(1,i)} : \phi_{(1,i+1)} ]$ \\ \hline
0 & 18 & 1 & $p$ \\
1 & 18 & 2 & $p^2$ \\
2 & 12 & 4 & $p^4$ \\ 
3 & 4 & 2 & $p^2$ \\
4 & 1 & 1 & $p$ \\ 
5 & 0 & 0 & 1 
\end{tabular}
\caption{The Jacobson radical of $\Adj(\circ)$ adds four new subgroups.}\label{tab:char}
\end{table}

\begin{table}[h]
\centering
\begin{tabular}{cccc}
Maximal Index & Origin of Subgroup & Order & Subgroup \\ \hline
$(1,0)$ & $G$ & $p^{13}$ & $\langle g_1, ..., g_{13} \rangle$ \\
$(1,1)$ & $J$ & $p^{12}$ & $\langle g_1,...,g_9,g_{11},...,g_{13} \rangle$ \\
$(1,2)$ & $J^2$ & $p^{10}$ & $\langle g_1,...,g_5,g_8,g_9,g_{11},...,g_{13} \rangle$ \\
$(1,3)$ & $J^3$ & $p^6$ & $\langle g_5,g_8,g_9,g_{11},...,g_{13} \rangle$ \\
$(1,4)$ & $J^4$ & $p^4$ & $\langle g_9,g_{11},...,g_{13}\rangle$\\
$(2,1)$ & $G'$ & $p^3$ & $\langle g_{11},...,g_{13}\rangle$\\
$(2,4)$ & generated & $p$ & $\langle g_{13}\rangle$ 
\end{tabular}
\caption{The resulting filter has length 7; nearly a four-fold increase in characteristic subgroups.}\label{tab:filter}
\end{table}

\end{ex}

\begin{ex}
We consider $G=$ {\sc SmallGroup}(512, $3\, 000\, 000$), which is $p$-class 2 \cite{EO:512}.
There is a nontrivial derivation refinement which yields three new subgroups.
The generation phase produces an additional three more characteristic subgroups, so after one refinement we have eight nontrivial subgroups in the filter.
The resulting filter has an adjoint refinement which outputs one new subgroup. 
The generation phase cannot generate additional subgroups because our prefilter is a composition series, so it just updates the indices.
\end{ex}

\begin{ex}
Let $G$ be a Sylow $2$-subgroup of $S_{100}$; it has order $2^{97}$ and is $p$-class 32.
Randomly choosing between adjoint and derivation refinements, we run through nine iterations and add 40 new subgroups.
Our resulting filter, $\phi: \mathbb{N}^{10}\rightarrow 2^G$, is constructed in less than 2 minutes and has 72 nontrivial subgroups.
\end{ex}

\begin{ex}
We look at a random sample of 2,000 sections of the Sylow $3$-subgroups of classical groups with Lie rank 15. 
We record how many new subgroups were found, relative to how many we started with, and the time it took to construct these filters. 
Scatter plots of these data are seen in Figure \ref{fig:RandomGroups}.

\begin{figure}[h]
\centering
\subcaptionbox{ 
We report on the growth, relative to the $p$-class, of each of the refinements. For large groups in this sample, we often find refinements. In addition, we plot the linear regression given by least squares.}
{\resizebox{0.85\linewidth}{!}{\input{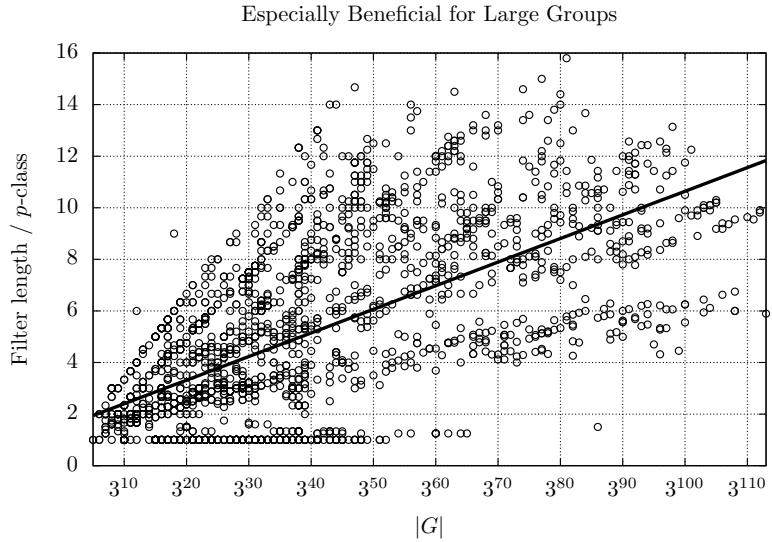}}}
\subcaptionbox{We record the total CPU time required to refine the filters. 
For the larger groups, as many as 20 iterations of {\sc Generate} were required.}
{\resizebox{0.85\linewidth}{!}{\input{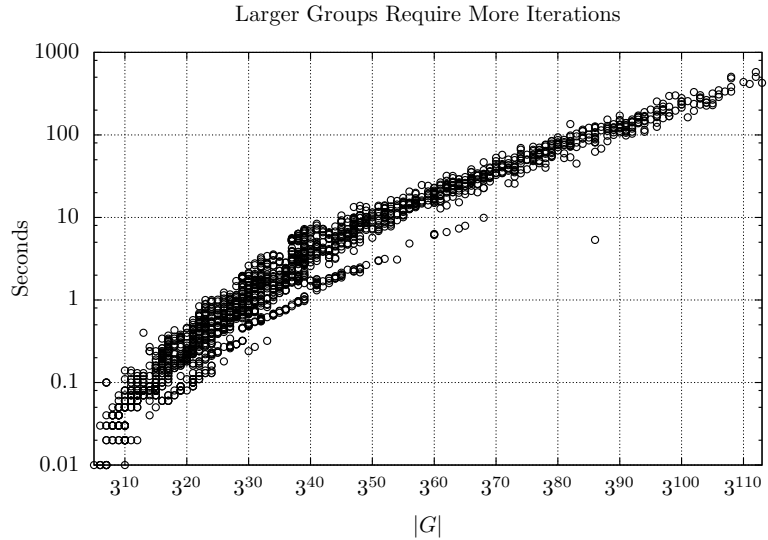}}}
\caption{We sample 2,000 sections of the Sylow 3-subgroups of groups of Lie type.
We refine filters until all the algebras in section \ref{refinements} are semisimple.}\label{fig:RandomGroups}
\end{figure}

\end{ex}

\section{Closing Remarks}

An obvious question might be to consider filters whose monoids are not totally ordered.
While there are benefits to this, there are some issues that have to be resolved.
For example, the associated Lie ring does not have to have the same order as the group. 
Indeed, it can be either larger or smaller.
Another issue to resolve is the data structure of a filter.
Our solution for filters with totally ordered monoids does not readily apply in the general context. 
One option might be to work with a finite monoid instead of $\mathbb{N}^d$.

A prototype {\sc Magma} implementation for filters is available from the author. 

\section*{Acknowledgements}
The author is thankful to James B. Wilson and Alexander Hulpke for insightful feedback and encouragement
and to the anonymous referee for helpful comments.

\bibliographystyle{alpha}
\bibliography{EfficientFilters}

\end{document}